\documentclass[twoside,12pt,a4paper,leqno]{amsart}

\usepackage{amsmath,amsfonts,amsthm,enumitem}
\usepackage[colorlinks=true,linkcolor=blue,citecolor=blue,%
filecolor=blue,urlcolor=blue,pageanchor=true,plainpages=false,%
linktocpage]{hyperref}

\numberwithin{equation}{section}
\newtheorem{thm}{Theorem}[section]

\newtheorem{cor}[thm]{Corollary}

\theoremstyle{definition}
\newtheorem{defn}[thm]{Definition}
\newtheorem{rem}[thm]{Remark}

\newcommand{\N}{\mathbb{N}}

\newcommand{\R}{\mathbb{R}}
\newcommand{\Rb}{\overline{\R}}

\newcommand{\degc}{\mathrm{ind}}
\makeatletter
\newcommand{\mylabel}[2]{#2\def\@currentlabel{#2}\label{#1}}
\makeatother

\newcommand{\hn}{$(N)$}
\newcommand{\hm}{$(M)$}
\newcommand{\hmc}{$(MC)$}
\newcommand{\vit}{$V\!I_t$}

\renewcommand{\rho}{\varrho}
\renewcommand{\theta}{\vartheta}
\newcommand{\au}[1]{\textsc{#1}}
\newcommand{\titleart}[1]{\textrm{#1}}
\newcommand{\jour}[1]{\textit{#1}}
\newcommand{\volart}[1]{\textbf{#1}}
\newcommand{\no}[1]{\textit{no.}~{#1}}

\begin{document}


\title[Quasilinear elliptic 
equations]{Quasilinear elliptic equations with natural growth and
quasilinear elliptic equations with singular drift}

\author{Marco Degiovanni \and Marco Marzocchi}
\address{Dipartimento di Matematica e Fisica\\
         Universit\`a Cattolica del Sacro Cuore\\
         Via dei Musei 41\\
         25121 Bre\-scia, Italy}
\email{marco.degiovanni@unicatt.it, marco.marzocchi@unicatt.it}
\thanks{The first author is member of the 
        Gruppo Nazionale per l'Analisi Matematica, la Probabilit\`a
				e le loro Applicazioni (GNAMPA) of the 
				Istituto Nazionale di Alta Matematica (INdAM).}





%
\begin{abstract}
We prove an existence result for a quasilinear elliptic
equation satisfying natural growth conditions.
As a consequence, we deduce an existence result for
a quasilinear elliptic equation containing a singular
drift.
A key tool, in the proof, is the study of an auxiliary
variational inequality playing the role of ``natural
constraint''.
\end{abstract}
\maketitle 


\section{Introduction}
Consider the quasilinear elliptic problem
\begin{equation}
\label{eq:qe}
\left\{
\begin{array}{ll}
- \mathrm{div}\left[a(x,u,\nabla u)\right]
+ b(x,u,\nabla u) = 0
&\qquad\text{in $\Omega$}\,,\\
\noalign{\medskip}
u=0
&\qquad\text{on $\partial\Omega$}\,,
\end{array}
\right.
\end{equation}
where $\Omega$ is a bounded and open subset of $\R^n$ and
\[
a:\Omega\times
\left(\R\times\R^n\right)\rightarrow \R^n\,,\qquad
b:\Omega\times
\left(\R\times\R^n\right)\rightarrow \R
\]
are two Carath\'eodory functions
satisfying the \emph{natural growth conditions}
in the sense of~\cite{giaquinta1983}.
More precisely, we assume that:
\emph{
\begin{enumerate}[align=parleft]
\item[\mylabel{hn}{\hn}]
there exist $1 < p < \infty$ and, for every $R>0$,
$\alpha^{(0)}_R\in L^1(\Omega)$,
$\alpha^{(1)}_R\in L^{p'}(\Omega)$, $\beta_R > 0$ 
and $\nu_R>0$ such that
\begin{gather*}
 |a(x,s,\xi)| \leq
\alpha^{(1)}_R(x) + \beta_R\,|\xi|^{p-1}\,,\qquad
|b(x,s,\xi)| \leq
\alpha^{(0)}_R(x) + \beta_R\,|\xi|^{p}\,,\\
\noalign{\medskip}
 a(x,s,\xi)\cdot\xi \geq
\nu_R\,|\xi|^{p} - \alpha^{(0)}_R(x)\,,
\end{gather*}
for a.e. $x\in\Omega$ and every $s\in\R$, $\xi\in\R^n$
with $|s|\leq R$;
such a $p$ is clearly unique;
\item[\mylabel{hm}{\hm}]
we have
\[
\left[a(x,s,\xi)-a(x,s,\hat{\xi})\right]
\cdot(\xi-\hat{\xi}) > 0
\]
for a.e. $x\in\Omega$ and every $s\in\R$, $\xi,\hat{\xi}\in\R^n$
with $\xi\neq\hat{\xi}$.
\end{enumerate}
}
\begin{defn}
We say that $w\in W^{1,p}_{loc}(\Omega)\cap L^\infty_{loc}(\Omega)$
is a \emph{supersolution} (resp. \emph{subsolution}) of the equation
\begin{equation}
\label{eq:int}
- \mathrm{div}\left[a(x,u,\nabla u)\right]
+ b(x,u,\nabla u) = 0
\qquad\text{in $\Omega$}\,,
\end{equation}
if 
\begin{multline*}
\int_\Omega \left[a(x,w,\nabla w)\cdot\nabla v + 
b(x,w,\nabla w) v\right]\,dx \geq 0
\qquad\text{(resp. $\leq 0$)} \\
\qquad\text{for every $v\in C^\infty_c(\Omega)$ with $v\geq 0$}\,.
\end{multline*}
\end{defn}
Let us state our main result.
\begin{thm}
\label{thm:main}
Assume there exist 
$\underline{u},\overline{u}\in 
W^{1,p}_{loc}(\Omega)\cap L^\infty(\Omega)$
and $u_0\in W^{1,p}_0(\Omega)\cap L^\infty(\Omega)$
such that $\underline{u}$ is a subsolution of~\eqref{eq:int},
$\overline{u}$ is a supersolution of~\eqref{eq:int}
and $\underline{u}\leq u_0 \leq \overline{u}$ a.e. in $\Omega$.
\par
Then there exists $u\in W^{1,p}_0(\Omega)\cap L^\infty(\Omega)$
satisfying $\underline{u}\leq u \leq \overline{u}$ a.e. in 
$\Omega$ and~\eqref{eq:qe} in a weak sense, namely
\[
\int_\Omega \left[a(x,u,\nabla u)\cdot\nabla v + 
b(x,u,\nabla u) v\right]\,dx = 0
\qquad\text{for every 
$v\in W^{1,p}_0(\Omega)\cap L^\infty(\Omega)$}\,.
\]
\end{thm}
\begin{rem}
The previous result should be compared 
with~\cite[Theorem]{hess1978} 
and~\cite[Th\'eo\-r\`e\-me~2.1]{boccardo_murat_puel1984}.
The main feature is that in~\cite{hess1978} a growth
condition of the form
\[
|b(x,s,\xi)| \leq
\alpha^{(0)}_R(x) + \beta_R\,|\xi|^{p-\varepsilon}
\]
with $\alpha^{(0)}_R \in L^1(\Omega)$ 
and $\varepsilon>0$ is required, so that the natural
growth of order $p$ in $\xi$ is not allowed.
\par
On the other hand, in~\cite{boccardo_murat_puel1984}
it is assumed that
\[
|b(x,s,\xi)| \leq
\beta_R\,\left(1 + |\xi|^{p}\right)
\]
and the term $\alpha^{(0)}_R \in L^1(\Omega)$
is not permitted (see also the remarks 
in~\cite{boccardo2018}).
\par
Here we take advantage of the framework
of~\cite{almi_degiovanni2013, degiovanni_pluda2017} 
to allow the condition
\[
|b(x,s,\xi)| \leq
\alpha^{(0)}_R(x) + \beta_R\,|\xi|^{p}\,,
\]
which seems to be the most general to guarantee
that $b(x,u,\nabla u)\in L^1(\Omega)$ whenever
$u\in W^{1,p}_0(\Omega)\cap L^\infty(\Omega)$.
\par
Just this level of generality will allow us to treat,
as a particular case, a problem with singular
drift, as we will see in the next result.
\end{rem}
\begin{rem}
A more general equation of the form
\[
- \mathrm{div}\left[a(x,u,\nabla u)\right]
+ b(x,u,\nabla u) = f_0 - \mathrm{div}\,f_1\,,
\]
with $f_0\in L^1(\Omega)$ and 
$f_1\in L^{p'}(\Omega;\R^n)$, can be easily reduced
to our case by setting
\begin{alignat*}{3}
&\check{a}(x,s,\xi) &&= a(x,s,\xi) - f_1(x)\,,\\
&\check{b}(x,s,\xi) &&= b(x,s,\xi) - f_0(x)\,.
\end{alignat*}
Of course, the key point is the existence of
\emph{bounded} super/subsolutions.
\end{rem}
\begin{cor}
\label{cor:main}
Assume that $a$ satisfies the further condition:
\[
a(x,s,0) = 0
\qquad\text{for a.e. $x\in\Omega$ and every $s\in\R$}\,.
\]
Let $0 < q \leq p$, $r>0$ and let 
$b_1 \in L^{\frac{p}{p-q}}(\Omega;\R^n)$\footnote{We
mean $b_1 \in L^\infty(\Omega;\R^n)$ if $q=p$.}
and $b_0, f \in L^1(\Omega)$ be such that:
\begin{equation}
\label{eq:Q}
\text{there exists $Q\geq 0$ satisfying
$|f(x)| \leq Q \, b_0(x)$
for a.e. $x\in\Omega$}\,.
\end{equation}
\indent
Then there exists $u\in W^{1,p}_0(\Omega)\cap L^\infty(\Omega)$
satisfying 
\[
- Q^{1/r} \leq u \leq Q^{1/r}
\qquad\text{a.e. in $\Omega$}
\]
and
\begin{multline*}
\int_\Omega \left[a(x,u,\nabla u)\cdot\nabla v + 
b_1\cdot\left(|\nabla u|^{q-1}\nabla u\right) v
+ b_0 |u|^{r-1}u v\right]\,dx = \int_\Omega v f\,dx \\
\qquad\text{for every 
$v\in W^{1,p}_0(\Omega)\cap L^\infty(\Omega)$}\,.
\end{multline*}
\end{cor}
\begin{rem}
Assumption~\eqref{eq:Q} is in particular satisfied if
$f\in L^\infty(\Omega)$ and $b_0(x) \geq \underline{b}>0$.
The previous corollary extends the results 
of~\cite{boccardo2018}, devoted to cases in which
the principal part of the equation is linear 
(see also, when the equation is fully linear, 
the paper~\cite{kim_kim2015}).
The technique of~\cite{boccardo2018} is based on a duality 
approach which seems not to be easily adaptable when the 
principal part of the equation is not linear.
\end{rem}
Concerning equations where condition~\eqref{eq:Q} is assumed,
Theorem~\ref{thm:main} allows us also to prove
the next corollary, which slightly generalizes
some results of~\cite{arcoya_boccardo2015} obtained by a
different technique.
\begin{cor}
\label{cor:int}
Assume that $a$ and $b$ satisfy the further condition:
\[
a(x,s,0) = 0\,,\qquad b(x,s,0) = 0
\qquad\text{for a.e. $x\in\Omega$ and every $s\in\R$}\,.
\]
Let $b_0, f \in L^1(\Omega)$ be such 
that~\eqref{eq:Q} holds with $Q>0$ and let 
$g:\R\rightarrow\R$ be a continuous function such that
\[
\lim_{s\to -\infty}\,g(s) = -\infty\,,\qquad
\lim_{s\to +\infty}\,g(s) = +\infty\,.
\]
\indent
Then there exists $u\in W^{1,p}_0(\Omega)\cap L^\infty(\Omega)$
satisfying 
\begin{multline*}
\int_\Omega \left[a(x,u,\nabla u)\cdot\nabla v + 
b(x,u,\nabla u)v + b_0 g(u) v\right]\,dx 
= \int_\Omega v f\,dx \\
\qquad\text{for every 
$v\in W^{1,p}_0(\Omega)\cap L^\infty(\Omega)$}
\end{multline*}
and $\underline A \leq u \leq \overline A$ a.e. in $\Omega$,
provided that $\underline A \leq 0 \leq \overline A$ and
$g(\underline A)\leq -Q$, $g(\overline A)\geq Q$.
\end{cor}
The proof of Theorem~\ref{thm:main} is based on the study 
of an auxiliary variational inequality, which plays the role 
of ``natural constraint'', in the sense that the solutions
of the variational inequality are automatically solutions
of the equation.
This kind of device appears many times in the 
literature and goes back, to our knowledge,
to~\cite{deuel_hess1976}.
A variant can be found in~\cite[Theorem~3.3]{marino1989} 
and~\cite[Theorem~2.3]{passaseo1989}
(see also~\cite{degiovanni_pluda2017, saccon2014}).
%


\section{Parametric quasilinear elliptic variational inequalities
with natural growth conditions}
\label{sect:qevin}
Throughout this section, we still consider two Carath\'eodory
functions $a$, $b$ satisfying~\ref{hn} and~\ref{hm}
and, moreover, a $p$-quasi upper semicontinuous 
function $\underline{u}:\Omega\rightarrow\Rb$
and a $p$-quasi lower semicontinuous 
function $\overline{u}:\Omega\rightarrow\Rb$.
It is well known that every $u\in W^{1,p}_0(\Omega)$ admits a 
Borel and $p$-quasi continuous representative $\tilde{u}$,
defined up to a set of null $p$-capacity, which we still 
denote by $u$ (see e.g.~\cite{dalmaso1983}).
\par
For every $t\in[0,1]$, we set
\begin{alignat*}{3}
&\underline{u}_t &&= \underline{u}-t\,,\\
&\overline{u}_t &&= \overline{u}+t\,,\\
&K_t&&=\left\{u\in W^{1,p}_0(\Omega)\cap L^{\infty}(\Omega):
\,\,\text{$\underline{u}_t\leq u \leq\overline{u}_t$
\,\,$p$-q.e. in $\Omega$}\right\}\,.
\end{alignat*}
We aim to consider the solutions $u$ of the parametric 
variational inequality
\begin{equation}
\label{eq:qevit}
\tag{\vit}
\begin{cases}
u\in K_t\,,\\
\noalign{\medskip}
\displaystyle{
\int_{\Omega}
\bigl[a(x,u,\nabla u)\cdot\nabla (v-u) +
b(x,u,\nabla u)\,(v-u)\bigr]\,dx \geq 0}\\
\noalign{\medskip}
\qquad\qquad\qquad\qquad\qquad\qquad\qquad\qquad\qquad\qquad
\text{for every $v\in K_t$}\,.
\end{cases}
\end{equation}
\begin{thm} 
\label{thm:strongcompactness}
Assume that $\underline{u}$, $\overline{u}$ are bounded and that
there exists $u_0\in W^{1,p}_0(\Omega)\cap L^\infty(\Omega)$
such that $\underline{u} \leq u_0 \leq \overline{u}$ 
$p$-q.e. in $\Omega$.
\par
Then the following facts hold:
\begin{enumerate}
\item[$(a)$] 
for every $t\in[0,1]$, there exists a solution $u$
of~\eqref{eq:qevit};
\item[$(b)$] 
the set
\[
\left\{ (u,t)\in
\left(W^{1,p}_0(\Omega)\cap L^{\infty}(\Omega)\right) 
\times [0,1]: \,\,
\text{$u$ is a solution of~\eqref{eq:qevit}} \right\}
\]
is strongly compact in $W^{1,p}_0(\Omega)\times[0,1]$.
\end{enumerate}
\end{thm}
\begin{proof}
We aim to apply the results of~\cite{degiovanni_pluda2017}.
Let us denote by $Z^{tot}_t$ the set of solutions 
of~\eqref{eq:qevit}.
Since $K_t\neq\emptyset$, 
from~\cite[Theorem~5.10]{degiovanni_pluda2017}
we infer that $\degc(Z^{tot}_t) = 1$.
Then~\cite[Theorem~5.6]{degiovanni_pluda2017} implies that
$Z^{tot}_t\neq\emptyset$ and assertion~$(a)$ follows..
\par
To prove assertion~$(b)$, we want to apply $(a)$ 
of~\cite[Theorem $5.8$]{degiovanni_pluda2017}. 
The only assumption we need to check is the\ continuity
of $\left\{t\mapsto K_t\right\}$ with respect to the Mosco
convergence~\hmc.
\par
Let $t_k\to t$ in $[0,1]$ and $(v_k)$ be a sequence weakly 
convergent to $v$ in $W^{1,p}_0(\Omega)$ with $v_k\in K_{t_k}$.
Up to a subsequence, we may assume that $(t_k)$ is monotone.
\par
If  $(t_k)$ is increasing, we have $v_{t_k} \in K_t$ for every 
$k\in\N$. 
Since $K_t$ is weakly closed, it follows that $v\in K_t$.
If  $(t_k)$ is decreasing, for every fixed $h\in\N$ it is 
$v_k\in K_{t_h}$ for every $k\geq h$, whence $v\in K_{t_h}$, 
namely
\[
\underline{u} - t_h \leq v \leq \overline u + t_h
\qquad\text{$p$-q.e. in $\Omega$}\,.
\] 
Since $t_h$ is converging to $t$, it follows that
\[
\underline{u} - t \leq v \leq \overline u + t
\qquad\text{$p$-q.e. in $\Omega$}\,,
\]
namely $v \in K_t$.
\par
Let now $t_k\to t$ in $[0,1]$ and $v\in K_t$. 
As before, up to a subsequence we may assume that $(t_k)$ 
is monotone.
\par
If $(t_k)$ is decreasing, we set $v_k=v$ for every $k\in\N$
and of course $(v_k)$ is strongly convergent to $v$
in $W^{1,p}_0(\Omega)$ with $v_k\in K_{t_k}$.
If $(t_k)$ is increasing, we set
\[
v_k = u_0 + (v - u_0 - t + t_k)^+ - (v - u_0 + t - t_k)^- \,.
\]
Since it is easily seen that $(v_k)$ converges to $v$ 
strongly in $W^{1,p}_0(\Omega)$, we have only to check 
that $v_k \in K_{t_k}$ eventually as $k\to\infty$.
If
\[
v(x) - u_0(x) - t + t_k > 0\,,
\]
we have
\[
v(x) - u_0(x) + t - t_k \geq
v(x) - u_0(x) - t + t_k > 0\,,
\]
whence
\[
v_k(x) = v(x) - t + t_k \geq u_0(x)\,.
\]
It follows
\[
\underline{u}(x) - t_k \leq u_0(x) 
\leq v_k(x) = v(x) - t + t_k
\leq \overline{u}(x) + t_k\,.
\]
If 
\[
v(x) - u_0(x) + t - t_k < 0\,,
\]
we infer in a similar way that
\[
\underline{u}(x) - t_k \leq
v_k(x) \leq \overline{u}(x) + t_k\,.
\]
Otherwise $v_k(x) = u_0(x)$,
which yields the same conclusion.
Therefore $v_k \in K_{t_k}$ and the proof is complete.
\end{proof}
%


\section{Solutions of equations versus solutions of variational
inequalities}
\label{sect:eqvareq}
Throughout this section, $\widehat{\Omega}$ will denote an
open subset of $\R^n$ and
\[
\hat{a}:\widehat{\Omega}\times
\left(\R\times\R^n\right)\rightarrow \R^n\,,\qquad
\hat{b}:\widehat{\Omega}\times
\left(\R\times\R^n\right)\rightarrow \R
\]
two Carath\'eodory functions such that:
\emph{
\begin{enumerate}[align=parleft]
\item[$(i)$]
there exist $1 < p < \infty$ and, for every compact
subset $C$ of $\widehat{\Omega}$ and every $R>0$,
$\alpha^{(0)}_{C,R}\in L^1(C)$,
$\alpha^{(1)}_{C,R}\in L^{p'}(C)$ and 
$\beta_{C,R} \geq 0$ such that
\begin{align*}
& |\hat{a}(x,s,\xi)| \leq
\alpha^{(1)}_{C,R}(x) + \beta_{C,R}\,|\xi|^{p-1}\,,\\
\noalign{\medskip}
& |\hat{b}(x,s,\xi)| \leq
\alpha^{(0)}_{C,R}(x) + \beta_{C,R}\,|\xi|^{p}\,,
\end{align*}
for a.e. $x\in C$ and every $s\in\R$, $\xi\in\R^n$
with $|s|\leq R$;
\item[$(ii)$]
we have
\[
\left[\hat{a}(x,s,\xi)-\hat{a}(x,s,\hat{\xi})\right]
\cdot(\xi-\hat{\xi}) \geq 0
\]
for a.e. $x\in\widehat{\Omega}$ and every $s\in\R$, 
$\xi,\hat{\xi}\in\R^n$.
\end{enumerate}
}
We also denote by $L^\infty_c(\widehat{\Omega})$ the set
of $v$'s in $L^\infty(\widehat{\Omega})$ vanishing a.e.
outside some compact subset of $\widehat{\Omega}$.
\begin{defn}
We say that 
$w\in W^{1,p}_{loc}(\widehat{\Omega})\cap 
L^\infty_{loc}(\widehat{\Omega})$ is a \emph{supersolution} 
(resp. \emph{subsolution}) of the equation
\begin{equation}
\label{eq:intl}
- \mathrm{div}\left[\hat{a}(x,u,\nabla u)\right]
+ \hat{b}(x,u,\nabla u) = 0
\qquad\text{in $\widehat{\Omega}$}\,,
\end{equation}
if 
\begin{multline*}
\int_{\widehat{\Omega}} 
\left[\hat{a}(x,w,\nabla w)\cdot\nabla v + 
\hat{b}(x,w,\nabla w) v\right]\,dx \geq 0
\qquad\text{(resp. $\leq 0$)} \\
\qquad\text{for every $v\in C^\infty_c(\widehat{\Omega})$ 
with $v\geq 0$}\,.
\end{multline*}
\end{defn}
\begin{thm} 
\label{thm:eqvareq} 
Let $\underline u, u, \overline u\in 
W^{1,p}_{loc}(\widehat{\Omega})\cap 
L^{\infty}_{loc}(\widehat{\Omega})$ be such that
$\underline{u}$ is a subsolution of~\eqref{eq:intl},
$\overline{u}$ is a supersolution of~\eqref{eq:intl}
and
\[
\begin{cases}
u\in \widehat{K}\,,\\
\noalign{\medskip}
\displaystyle{
\int_{\widehat{\Omega}}
\left[\hat{a}(x,u,\nabla u)\cdot\nabla (v-u) +
\hat{b}(x,u,\nabla u)\,(v-u)\right]\,dx \geq 0}\\
\noalign{\medskip}
~\qquad\qquad\qquad\qquad\qquad\qquad\qquad
\text{for every $v\in \widehat{K}$ 
with $(v-u)\in L^{\infty}_c(\widehat{\Omega})$}\,,
\end{cases}
\]
where
\[
\widehat{K} = \left\{v\in W^{1,p}_{loc}(\widehat{\Omega})
\cap L^{\infty}_{loc}(\widehat{\Omega}):\,\,
\text{$\underline{u}\leq v \leq \overline{u}$ a.e. in 
$\widehat{\Omega}$}
\right\}\,.
\]
Suppose also that:
\begin{enumerate}
\item[$(iii)$]
for every compact subset $C$ of $\widehat{\Omega}$, there 
exist $r_C>0$ and $\gamma_C\in L^{p'}(C)$ such that
\begin{alignat*}{3}
&\left|\hat{a}(x,s,\nabla\underline u(x))
- \hat{a}(x,\underline u(x),\nabla\underline u(x))\right|
&&\leq \gamma_C(x)(s - \underline u(x))\,,\\
&\left|\hat{a}(x,\overline u(x),\nabla\overline u(x))
- \hat{a}(x,\sigma,\nabla\overline u(x))\right|
&&\leq \gamma_C(x)(\overline u(x)-\sigma) \,,
\end{alignat*}
for a.e. $x\in C$ and every $s,\sigma\in\R$ with
$\underline u(x)\leq s \leq \underline u(x) + r_C$ and 
$\overline u(x) - r_C \leq \sigma \leq \overline u(x)$.
\end{enumerate}
\par
Then we have 
\[
\int_{\widehat{\Omega}} 
\left[\hat{a}(x,u,\nabla u)\cdot\nabla v + 
\hat{b}(x,u,\nabla u) v\right]\,dx = 0
\qquad\text{for every 
$v\in W^{1,p}_0(\widehat{\Omega})
\cap L^{\infty}_c(\widehat{\Omega})$} \,.
\]
\end{thm}
\begin{proof}
Let $v\in C^{\infty}_c(\widehat{\Omega})$ with $v\geq 0$, 
let $t>0$ and let
\[
u_t = \min\left\{ u+tv,\overline u\right\}\,.
\]
Clearly $u_t\in \widehat{K}$ and 
$(u_t-u)\in L^{\infty}_c(\widehat{\Omega})$, whence
\[
\int_{\widehat{\Omega}}
\left[\hat{a}(x,u,\nabla u)\cdot\nabla (u_t-u) +
\hat{b}(x,u,\nabla u)\,(u_t-u)\right]\,dx \geq 0 \,.
\]
Taking into account assumption~$(ii)$, we get
\[
\int_{\widehat{\Omega}}
\left[\hat{a}(x,u,\nabla u_t)\cdot\nabla (u_t-u) +
\hat{b}(x,u,\nabla u)\,(u_t-u)\right]\,dx \geq 0 \,,
\]
namely
\begin{multline*}
\int_{\widehat{\Omega}}
\left[\hat{a}(x,u_t,\nabla u_t)\cdot\nabla (u_t- u - tv) 
+ \hat{b}(x,u_t,\nabla u_t)\,(u_t- u - tv)\right]\,dx  \\
+ t\int_{\widehat{\Omega}} 
\left[\hat{a}(x,u_t,\nabla u_t)\cdot\nabla v
+ \hat{b}(x,u_t,\nabla u_t) v\right]\,dx 
\qquad\qquad\qquad\qquad~ \\
~\qquad\qquad
\geq \int_{\widehat{\Omega}}
\left[\hat{a}(x,u_t,\nabla u_t) - \hat{a}(x,u,\nabla u_t)\right]
\cdot\nabla (u_t-u-tv)\,dx \\ 
~\qquad\qquad\qquad\qquad
+ t \int_{\widehat{\Omega}}
\left[\hat{a}(x,u_t,\nabla u_t) - \hat{a}(x,u,\nabla u_t)\right]
\cdot\nabla v\,dx \\
+ \int_{\widehat{\Omega}}
\left[\hat{b}(x,u_t,\nabla u_t) - \hat{b}(x,u,\nabla u)\right]
(u_t-u)\,dx \,.
\end{multline*}
Since $u_t = \overline u$ where $u_t- u - tv\neq 0$, we have
\begin{multline*}
\int_{\widehat{\Omega}}
\left[\hat{a}(x,u_t,\nabla u_t)\cdot\nabla (u_t- u - tv) +
\hat{b}(x,u_t,\nabla u_t)\,(u_t-u - tv)\right]\,dx  \\
= \int_{\widehat{\Omega}}
\left[\hat{a}(x,\overline u,\nabla \overline u)
\cdot\nabla (u_t- u - tv) 
+ \hat{b}(x,\overline u,\nabla \overline u)\,
(u_t-u - tv)\right]\,dx \leq 0 \,,
\end{multline*}
as $\overline u$ is a supersolution of~\eqref{eq:intl} and
$u_t- u - tv \leq 0$.
That leads to the final inequality
\begin{multline*}
\int_{\widehat{\Omega}} 
\left[\hat{a}(x,u_t,\nabla u_t)\cdot\nabla v
+ \hat{b}(x,u_t,\nabla u_t) v\right]\,dx  \\
~\qquad\qquad
\geq \int_{\widehat{\Omega}}
\frac{\hat{a}(x,\overline u,\nabla \overline u) 
- \hat{a}(x,u,\nabla \overline u)}{t} 
\cdot\nabla (u_t-u-tv)\,dx \\ 
~\qquad\qquad\qquad\qquad
+  \int_{\widehat{\Omega}}
\left[\hat{a}(x,u_t,\nabla u_t) - \hat{a}(x,u,\nabla u_t)\right]
\cdot\nabla v\,dx \\
+ \int_{\widehat{\Omega}}
\left[\hat{b}(x,u_t,\nabla u_t) - \hat{b}(x,u,\nabla u)\right]
\frac{u_t-u}{t}\,dx \,.
\end{multline*}
Since $\left|\dfrac{u_t-u}{t}\right| \leq v$, from
assumption~$(i)$ it follows that
\begin{align*}
&\lim_{t\to 0^+}\,
\int_{\widehat{\Omega}} 
\left[\hat{a}(x,u_t,\nabla u_t)\cdot\nabla v
+ \hat{b}(x,u_t,\nabla u_t) v\right]\,dx \\
& \qquad\qquad\qquad\qquad\qquad
= \int_{\widehat{\Omega}} 
\left[\hat{a}(x,u,\nabla u)\cdot\nabla v
+ \hat{b}(x,u,\nabla u) v\right]\,dx \,,\\
&\lim_{t\to 0^+}\,
\int_{\widehat{\Omega}}
\left[\hat{a}(x,u_t,\nabla u_t) - \hat{a}(x,u,\nabla u_t)\right]
\cdot\nabla v\,dx = 0\,,\\
&\lim_{t\to 0^+}\,
\int_{\widehat{\Omega}}
\left[\hat{b}(x,u_t,\nabla u_t) - \hat{b}(x,u,\nabla u)\right]
\frac{u_t-u}{t}\,dx = 0\,.
\end{align*}
Let now $C$ be a compact subset of $\widehat{\Omega}$ such that
$v=0$ outside $C$ and let $r_C>0$ and $\gamma_C\in L^{p'}(C)$ 
be as in assumption~$(iii)$.
Without loss of generality, we may assume that
$tv \leq r_C$ on $C$.
Then, since $0 \leq \overline u - u < tv \leq r_C$ where 
$u_t- u - tv\neq 0$, we get
\[
\left|\frac{\hat{a}(x,\overline u,\nabla \overline u) 
- \hat{a}(x,u,\nabla \overline u)}{t} 
\cdot\nabla (u_t-u-tv)\right| \leq
\gamma_C v \left|\nabla (\overline u-u-tv)\right|\,.
\]
Again from assumption~$(i)$ we infer that
\[
\lim_{t\to 0^+}\,
\int_{\widehat{\Omega}}
\frac{\hat{a}(x,\overline u,\nabla \overline u) 
- \hat{a}(x,u,\nabla \overline u)}{t} 
\cdot\nabla (u_t-u-tv)\,dx = 0\,.
\]
Therefore we have
\[
\int_{\widehat{\Omega}} 
\left[\hat{a}(x,u,\nabla u)\cdot\nabla v
+ \hat{b}(x,u,\nabla u) v\right]\,dx \geq 0
\]
for every $v\in C^{\infty}_c(\widehat{\Omega})$ with $v\geq 0$.
\par
If $v\in C^{\infty}_c(\widehat{\Omega})$ with $v\leq 0$, 
we consider $t>0$ and
\[
u_t = \max\{ u+tv,\underline u\}\,.
\]
Arguing in a similar way, we get
\[
\int_{\widehat{\Omega}} 
\left[\hat{a}(x,u,\nabla u)\cdot\nabla v
+ \hat{b}(x,u,\nabla u) v\right]\,dx \geq 0
\]
for every $v\in C^{\infty}_c(\widehat{\Omega})$ with $v\leq 0$.
\par
It follows
\[
\int_{\widehat{\Omega}} 
\left[\hat{a}(x,u,\nabla u)\cdot\nabla v
+ \hat{b}(x,u,\nabla u) v\right]\,dx \geq 0
\]
for every $v\in C^{\infty}_c(\widehat{\Omega})$, whence
the equality, as we can swap $v$ with $-v$.
\par
The case  $v\in W^{1,p}_0(\widehat{\Omega})
\cap L^{\infty}_c(\widehat{\Omega})$ can be treated
by a standard approximation argument.
\end{proof}
\begin{rem}
Assumption~$(iii)$ is obviously satisfied in the
following cases:
\begin{enumerate}
\item[$(a)$]
the function $\hat{a}(x,s,\xi)$ is independent
of $s$;
\item[$(b)$]
we have $\hat{a}(x,s,0)=0$ and 
$\underline u, \overline u$ are constant.
\end{enumerate}
On the other hand, we do not know whether
Theorem~\ref{thm:eqvareq} holds true without
assumption~$(iii)$.
\end{rem}
%


\section{Proof of the results stated in the Introduction}
\label{sect:proofs}
\noindent
\emph{Proof of Theorem~\ref{thm:main}.}
\par\noindent
It is easily seen that 
\begin{alignat*}{3}
&\check a(x,s,\xi) &&= 
a(x, \min\{\max\{s,\underline u(x)\},\overline u(x)\},\xi) \,, \\
&\check b(x,s,\xi) &&= 
b(x, \min\{\max\{s,\underline u(x)\},\overline u(x)\},\xi) 
\end{alignat*}
are still two Carath\'eodory functions satisfying 
assumptions~\ref{hn} and~\ref{hm} and
$\underline u$, $\overline u$
are respectively a subsolution and a supersolution 
of
\[
- \mathrm{div}\left[\check a(x,u,\nabla u)\right]
+ \check b(x,u,\nabla u) = 0
\qquad\text{in $\Omega$} \,.
\]
On the other hand, if $u$ satisfies the assertion
with respect to $\check a$ and $\check b$, then it does 
the same with respect to $a$ and $b$, 
as $\underline{u} \leq u \leq \overline{u}$ a.e.
in $\Omega$.
Therefore, without loss of generality, we may assume that
$a(x,s,\xi)$ and $b(x,s,\xi)$ are independent of $s$ for
$s\leq \underline{u}(x)$ and for $s\geq \overline{u}(x)$.
\par
As in Section~\ref{sect:qevin}, for every $t\in[0,1]$ we set
\[
\underline u_t = \underline u - t\,, \qquad 
\overline u_t = \overline u + t \,.
\]
Then, assumptions~$(i)$ and~$(ii)$ of Section~\ref{sect:eqvareq}
are obviously satisfied and $\underline u_t$, $\overline u_t$
are respectively a subsolution and a supersolution 
of~\eqref{eq:int} for any $0\leq t\leq 1$.
If $t>0$, also the hypothesis~$(iii)$ of 
Theorem~\ref{thm:eqvareq} holds true, 
as $a(x,s,\xi)$ is independent of $s$ for 
$s \leq \underline u_t(x)+t$ and for $s\geq \overline u_t(x) - t$.
\par
From Theorem~\ref{thm:strongcompactness} we infer that 
for every $t\in[0,1]$ there exists a solution $u$
of~\eqref{eq:qevit} and that the set
\[
\left\{(u,t)\in\left(W^{1,p}_0(\Omega)\cap 
L^{\infty}(\Omega)\right) \times [0,1]:\,\,
\text{$u$ is a solution of~\eqref{eq:qevit}} \right\}
\]
is strongly compact in $W^{1,p}_0(\Omega)\times[0,1]$. 
\par
Let now, for every $m\geq 1$, $u_m$ be a solution 
of~\eqref{eq:qevit} with $t=1/m$. 
Then $(u_m)$ is bounded in $L^\infty(\Omega)$ and, up to a
subsequence, strongly convergent in $W^{1,p}_0(\Omega)$ to 
some $u$ satisfying~\eqref{eq:qevit} with $t=0$.
In particular, we have
$\underline{u} \leq u \leq \overline{u}$ a.e.
in $\Omega$.
\par
From Theorem~\ref{thm:eqvareq} we infer that
each $u_m$ actually satisfies
\[
\int_{\Omega} 
\left[a(x,u_m,\nabla u_m)\cdot\nabla v + 
b(x,u_m,\nabla u_m) v\right]\,dx = 0
\qquad\text{for every 
$v\in W^{1,p}_0(\Omega)\cap L^{\infty}_c(\Omega)$} \,.
\]
Going to the limit as $m\to\infty$, it easily follows that
\[
\int_{\Omega} 
\left[a(x,u,\nabla u)\cdot\nabla v + 
b(x,u,\nabla u) v\right]\,dx = 0
\qquad\text{for every 
$v\in W^{1,p}_0(\Omega)\cap L^{\infty}_c(\Omega)$} \,.
\]
By a standard density argument, the equation holds
for any $v\in W^{1,p}_0(\Omega)\cap L^\infty(\Omega)$.
\qed
\par\bigskip
\noindent
\emph{Proof of Corollary~\ref{cor:main}.}
\par\noindent
If $Q=0$, we have $f=0$ and the assertion is satisfied by $u=0$.
Assume that $Q>0$, so that $b_0(x)\geq 0$ a.e. in $\Omega$, and set
\[
b(x,s,\xi) = b_1(x)\cdot(|\xi|^{q-1}\xi)
+b_0(x)|s|^{r-1}s - f(x)\,.
\]
In the case $0 < q < p$ we have
\[
|b(x,s,\xi)| \leq
\frac{p-q}{p}\,|b_1(x)|^{\frac{p}{p-q}}
+ \frac{q}{p}\,|\xi|^p + R^r\,b_0(x) + |f(x)| \,,
\]
whenever $|s|\leq R$.
Therefore assumptions~\ref{hn} and~\ref{hm} are satisfied.
In the case $q=p$ we have $b_1\in L^\infty(\Omega;\R^n)$
and the argument is even simpler.
\par
On the other hand, if we set
\[
\underline{u}(x) = - Q^{1/r}\,,\qquad
\overline{u}(x) = Q^{1/r}\,,
\]
it follows that $a(x,\underline{u},\nabla\underline{u})=0$ and
\[
b(x,\underline{u},\nabla\underline{u}) =
- Q \, b_0 - f \leq - Q \, b_0 + Q\,b_0 =0\,,
\]
so that $\underline{u}$ is a subsolution
of~\eqref{eq:int}.
The proof that $\overline{u}$ is a supersolution
of~\eqref{eq:int} is similar.
\par
By Theorem~\ref{thm:main} the assertion follows.
\qed
\par\bigskip
\noindent
\emph{Proof of Corollary~\ref{cor:main}.}
\par\noindent
If we set
\[
\check{b}(x,s,\xi) = b(x,s,\xi) + b_0(x) g(s) - f(x)\,,
\]
it is easily seen that $a$ and $\check{b}$ still
satisfy assumptions~\ref{hn} and~\ref{hm}.
\par
Let $\underline A \leq 0 \leq \overline A$ be such that
\[
g(\underline A) \leq -Q\,,\qquad g(\overline A) \geq Q\,.
\]
Then, if we set
\[
\underline{u}(x) = \underline A\,,\qquad
\overline{u}(x) = \overline A\,,
\]
it follows that $a(x,\underline{u},\nabla\underline{u})=0$,
$b(x,\underline{u},\nabla\underline{u})=0$ and
\[
\check{b}(x,\underline{u},\nabla\underline{u}) =
b_0 g(\underline A) - f \leq - Q \, b_0 + Q\,b_0 =0\,,
\]
so that $\underline{u}$ is a subsolution of~\eqref{eq:int}
with $b$ replaced by $\check{b}$.
The proof that $\overline{u}$ is a supersolution
is similar.
\par
By Theorem~\ref{thm:main} the assertion follows.
\qed
%


%
\end{document}